\documentclass[12pt, reqno]{amsart}
\usepackage{amsmath, amsfonts, amssymb, url}
\usepackage{eucal}
\usepackage{mathrsfs}
\usepackage{latexsym}
\usepackage{cite}
\usepackage{cases}
\allowdisplaybreaks[4]
\newtheorem{thm}{Theorem}[section]

\newtheorem{lem}[thm]{Lemma}

\theoremstyle{remark}
\newtheorem{rem}[thm]{Remark}

\theoremstyle{definition}
\newtheorem{definition}[thm]{Definition}

\numberwithin{equation}{section}

\begin{document}
	\title[Dual $L_p$ John ellipsoids for general measures]
	{Dual $L_p$ John ellipsoids for general measures}
	
\author[A.-J. Li]{Ai-Jun Li}
\address[A.-J. Li]{School of Science,
	Zhejiang University of Science and Technology, Hangzhou 310023, China}
\email{liaijun72@163.com}	
	\author[B. Liu]{Baihe Liu}
\address[ B. Liu]{School of Science,
	Zhejiang University of Science and Technology, Hangzhou 310023, China} 
\email{13140528462@163.com} 
\author[Q. Huang]{Qingzhong Huang}
\address[Q. Huang]{College of data science, Jiaxing University, Jiaxing 314001,	China} \email{hqz376560571@163.com}

	\begin{abstract} 
{
The dual $L_p$ John ellipsoids form a family that includes the origin-centered L\"owner and Legendre ellipsoids. In this paper, we extend this family to general measures within the dual weighted $L_p$ Brunn--Minkowski theory. For measures with locally integrable densities that are continuous in a neighborhood of the boundary of the given star body, we establish a variational formula for $L_p$-harmonic radial combinations. The associated boundary integral leads to an optimization problem that remains meaningful under weaker assumptions on the boundary trace of the density. We prove the existence, uniqueness, and characterization of its solution. The resulting ellipsoids include the dual $L_p$ John ellipsoids for Gaussian measures. We also establish a monotonicity inequality, a measure inequality for homogeneous densities with its equality conditions on the weighted support, and an $L_p$ L\"owner inclusion.
}	\end{abstract}

	
	
	\thanks{The first author was supported by NSFC (Grant No. 12571146 and 12231006). The third author was supported by  NSFC (Grant No. 11701219) and the Qin Shen Scholar Program of Jiaxing University.}

	\maketitle

	\section{Introduction}

Throughout this paper, a convex body $K$ in the $n$-dimensional Euclidean
space $\mathbb{R}^n$ is a compact convex set with nonempty interior. Its support function $h_K:\mathbb R^n\rightarrow \mathbb{R}$  is defined by
\begin{equation}\label{e-6}
	h_K(x)=\max \{x\cdot y: y\in K\},
\end{equation} for each $x\in \mathbb{R}^n$, where $x\cdot y$ denotes the standard inner product of $x$ and $y$. Let
$\mathcal {K}_{o}^n$ denote the set of convex bodies in $\mathbb R^n$ that contain the origin $o$ in their interiors.
For $p\geq 1$, $\varepsilon> 0$, and $K,L\in\mathcal {K}_{o}^n$, the Minkowski-Firey  $L_p$-combination $K +_p \varepsilon \cdot L\in\mathcal {K}_{o}^n$
is defined in \cite{Firey} by  
\begin{equation}\label{lp}
	h^p_{K +_p \varepsilon\cdot L} (x)=h^p_K (x)+\varepsilon h^p_L(x),
\end{equation}
for all $x \in \mathbb{R}^n$. 
 {
In
} \cite{L93}, Lutwak derived the following variational formula for the Minkowski-Firey $L_p$-combination \eqref{lp}:
\begin{equation}\label{e-3}
	\lim_{\varepsilon\rightarrow
		0^+}\dfrac{V({
K+_{p}\varepsilon\cdot
}
		L)-V(K)}{\varepsilon}=\frac{n}{p} \int_{S^{n-1}}
	\Big(\frac{h_L(u)}{h_K(u)}\Big)^pdV_K(u),
\end{equation}
where  $V$ denotes  the volume (Lebesgue measure) in $\mathbb{R}^n$, $V_K$ is the
cone-volume measure of $K$ on the  unit sphere $S^{n-1}$  defined by $dV_K=\frac{1}{n}h_KdS_K$  and $S_K$ is the classical surface area measure of $K$ on $S^{n-1}$. {
This variational formula is one of the starting points of the $L_p$ Brunn--Minkowski theory. For further background and bibliographical references, we refer the reader to
} \cite[Chapter 9]{Schneider} and the references therein.

Based on the variational formula \eqref{e-3}, Lutwak, Yang, and Zhang
\cite{LYZJ} proposed the following optimization problem.

\noindent\textbf{Problem~$S_p$.} {\em Given a convex
	body~$K\in\mathcal {K}_{o}^n$
	{
and~$0<p<\infty$
}, find an origin-symmetric ellipsoid~$E$ which solves the following constrained
	extremal problem:
	\begin{equation*}\max_{E\in \mathcal{E}^n_e}V(E)\quad\textrm{subject to}\quad
		\frac{1}{V(K)}\int_{S^{n-1}}\Big(\frac{h_E(u)}{h_K(u)}\Big)^pdV_K(u)\le 1
		,\end{equation*} where $\mathcal{E}^n_e$ denotes the
	set of all origin-symmetric ellipsoids in $\mathbb R^n$. } \\

\noindent  Lutwak et al. \cite{LYZJ} proved that there exists a
unique origin-symmetric ellipsoid which solves Problem $S_p$. This
ellipsoid, denoted by $E_p K$, is usually called the $L_p$ John ellipsoid of $K$ and characterized in \cite{LYZJ} as follows:  $E_p K$ solves Problem $S_p$ {
for $0<p<\infty$ if and only
if for all $x \in \mathbb{R}^n$,
}
\begin{equation*}
	V(K)h_{E_p^* K}(x)^{2}=n\int_{S^{n-1}}|x \cdot u|^2 h_{E_p K}(u)^{-2}\Big(\frac{h_{E_p K}(u)}{h_K(u)}\Big)^pdV_K(u),
\end{equation*}
{
where $E^\ast_p K$ denotes the polar body of $E_p K$. The case $p=\infty$ is understood through the corresponding maximum constraint, rather than a power integral.
}

{
The dual Brunn--Minkowski theory was initiated by Lutwak \cite{L75} in the 1970s. In this theory, star bodies and radial combinations play roles analogous to those of convex bodies and Minkowski combinations in the classical theory.
}  {
Recall that a compact set $K\subseteq\mathbb R^n$ containing the origin $o$ is star-shaped with respect to $o$ if its intersection with each ray issuing from $o$ is a line segment containing $o$.
}  Its radial function
$\rho_K:\mathbb R^n \backslash \left\{o\right\}\rightarrow
[0,\infty)$ is defined by
\begin{equation}\label{JXHS}
	\rho_K(x)=\max\left\{ \lambda \geq 0 : \lambda x\in K\right\},
\end{equation}
for each $x\in\mathbb R^n \backslash \left\{o\right\}$. A star-shaped set $K$ is called a star body if the radial function restricted on $S^{n-1}$ is positive and continuous.  Let
$\mathcal {S}_{o}^n$ denote the set of star bodies in $\mathbb R^n$. For $p>0$, $\varepsilon>0$, and $K,L\in\mathcal {S}_{o}^n$, the $L_p$-harmonic radial combination $K\widetilde{+}_{-p}\varepsilon\cdot L\in\mathcal {S}_{o}^n$ is  defined by
\begin{equation}\label{prh}
	\rho_{K\widetilde{+}_{-p}\varepsilon\cdot L}^{-p}(x)=\rho_K^{-p} (x)+\varepsilon\rho_L^{-p} (x),
\end{equation}	
for all $x \in \mathbb{R}^n\backslash \{o\}$.  Analogous to \eqref{e-3}, Lutwak \cite{L96} also established a variational formula for the $L_p$-harmonic radial combination \eqref{prh}:
\begin{align}\label{Vp}
	\lim_{\varepsilon\rightarrow
		0^+}\frac{V(K\widetilde{+}_{-p}\varepsilon\cdot
		L)-V(K)}{\varepsilon}=-\frac{n}{p} \int_{S^{n-1}}\Big(\dfrac{\rho_L(u)}{\rho_K(u)}\Big)^{-p} d\widetilde{V}_{K}(u),
\end{align}
where we use the  dual conical measure $\widetilde{V}_{K}$ of $K$ on $S^{n-1}$  defined in \cite{LYZ10} by $d\widetilde{V}_{K}=\frac{1}{n}\rho_K^nd\sigma_{n-1}$  and $\sigma_{n-1}$ is the spherical Lebesgue measure on $S^{n-1}$. Although  \(\widetilde{V}_{K}\) and \(V_K\) have distinct expressions, their theoretical roots all lead back to the cone measure
 \cite{Barthe05,GM}. In recent years, the dual Brunn-Minkowski theory has received considerable attention. Some important geometric measures were introduced and studied. For example, the dual curvature measures  and the corresponding dual Minkowski problems (see, e.g., \cite{HLYZ,LYZ18,Zhao,BLYZZ}) have been extensively investigated. 

Corresponding to Problem $S_p$, Yu,  Leng, and  Wu \cite{YLW} proposed the following optimization problem (see also  \cite{BR}).

\noindent\textbf{Problem $\widetilde S_p$.} {\em Given a star
	body $K\in\mathcal {S}_{o}^n$  {
and $0<p<\infty$, find an origin-symmetric ellipsoid $E$ which solves the following constrained
	extremal problem:
}
	\begin{equation*}
		\min_{E\in \mathcal{E}^n_e}V(E)\quad\textrm{subject to}\quad
		\frac{1}{V(K)}\int_{S^{n-1}}\Big(\dfrac{\rho_E(u)}{\rho_K(u)}\Big)^{-p} d\widetilde{V}_{K}(u)\le 1.
	\end{equation*} 
} \\
\noindent Yu  et al. \cite{YLW} proved that there exists a unique origin-symmetric ellipsoid which solves Problem $\widetilde S_p$. This
ellipsoid, denoted by $\widetilde E_p K$, is called the dual $L_p$ John ellipsoid of $K$ and   characterized in \cite{YLW} as follows:  $\widetilde E_p K$ solves Problem $\widetilde S_p$ {
for $0<p<\infty$ if and only
if for all $x \in \mathbb{R}^n\backslash \{o\}$,
}
\begin{equation}\label{s4}
	V(K)\rho_{\widetilde E_p^* K}(x)^{-2}=n\int_{S^{n-1}}|x \cdot u|^2\rho_{\widetilde E_p K}(u)^{2}\Big(\dfrac{\rho_{\widetilde E_p K}(u)}{\rho_K(u)}\Big)^{-p} d\widetilde{V}_{K}(u),
\end{equation}
{
where $\widetilde E_p^* K$ denotes the polar body of $\widetilde E_p K$. At $p=\infty$, the constraint is interpreted as $K\subseteq E$.
}

Note that  $L_p$ John ellipsoids unify several  fundamental ellipsoids, such as {
the
classical John ellipsoid for origin-symmetric bodies \cite{J} ($p=\infty$)
}, the Petty ellipsoid \cite{PettyS} ($p=1$)
and the LYZ ellipsoid \cite{LYZel} ($p=2$, in some
sense dual to the Legendre ellipsoid of inertia in classical
mechanics \cite{MP}). {
Similarly, dual $L_p$ John ellipsoids unify the origin-centered L\"{o}wner ellipsoid ($p=\infty$) and
}
the Legendre ellipsoid ($p=2$). In particular, Lutwak, Yang, and Zhang \cite{LYZCR} showed that the inclusion between the LYZ ellipsoid and the Legendre ellipsoid is the geometric analogue of the Cramer-Rao inequality, a fundamental inequality in the information theory. Among other things, the John ellipsoid  is the unique ellipsoid of maximal volume contained in convex body $K$, whereas the L\"{o}wner ellipsoid is the unique ellipsoid of minimal volume containing $K$. Both ellipsoids are  investigated widely in the field of convex geometry and local theory of
Banach spaces, see, e.g., \cite{Ball-91, Barthe,Gia,Gruber,J,Lew,GiaM,F-Barthe,Alonso-S-2}.

{
These constructions have been extended in several directions. In particular,
} {
$L_p$ John ellipsoids and their duals have led to
}   logarithmic
John ellipsoids \cite{HX},  mixed $L_p$ John ellipsoids \cite{HXZ}, Orlicz-John ellipsoids \cite{ZX}, 
Orlicz-Legendre ellipsoids \cite{ZX2},  $(p, q)$-John ellipsoids \cite{LWZ,MWF},  sine ellipsoids \cite{LHX}, and other extensions (negative indices) \cite{Lv,LuX}.

Over the last two decades, various problems from the $L_p$ Brunn-Minkowski {
theory have been extended to
general measures. These extensions form part of
} the weighted $L_p$ Brunn-Minkowski theory (see, e.g., \cite{Hos1,Li,LRZ,LM,Mil,RX,Wu,Wu2,Zv,KL,FLMZI,FLMZII,Wu3,GWXY}). {
The corresponding optimization problem for $L_p$ John ellipsoids and general measures was formulated and solved by
}
Ai  and Huang \cite{AWHQ}. In light of the above, a natural question in the framework of   the  dual  weighted $L_p$ Brunn-Minkowski theory arises: what are dual $L_p$ John ellipsoids for general
measures?

{
We begin with the boundary measure that enters the optimization problem.
}  
\begin{definition}\label{mud}
Let $\mu$ be a Borel measure on $\mathbb{R}^n$ that has a locally integrable Radon-Nikodym derivative $g$ from $\mathbb{R}^n$ to $[0, \infty)$; i.e.,
$$
\frac{d\mu(x)}{dx} = g(x),\text{ with } g: \mathbb{R}^n \to [0, \infty), g \in L^1_{\text{loc}}(\mathbb{R}^n).
$$
{\noindent
Let $K\in\mathcal S_o^n$. Suppose that $\partial K$, up to a set of $(n-1)$-dimensional Hausdorff measure zero, lies in the Lebesgue set of $g$. At these points, $g(r_K(u))$ denotes the precise representative
\[
g(r_K(u))=\lim_{r\to0^+}\frac{1}{V(B(r_K(u),r))}\int_{B(r_K(u),r)}g(x)\,dx,
\]
and its value may be set to zero on the exceptional set. Here   $B(x,r)=\{y\in\mathbb R^n:|y-x|<r\}$ denotes the open
Euclidean ball centered at $x$ with radius $r>0$ and the radial map $r_K:S^{n-1}\rightarrow\partial K$ is  a continuous map  defined by $r_K(u)=\rho_K(u)u$ for $u\in S^{n-1}$.  Assume also that $g\circ r_K\in L^1(S^{n-1},\sigma_{n-1})$. The dual conical $\mu$-measure $\widetilde V_{\mu,K}$ is the finite Borel measure defined by
}  \begin{equation}\label{e-7}
		\widetilde{V}_{\mu,K}(\eta)=\frac{1}{n}\int_{\eta}\rho_K(u)^ng\big(r_K(u)\big)d\sigma_{n-1}(u),
	\end{equation}
for any Borel set $\eta\subset S^{n-1}$. 
\end{definition}

{
The following variational formula holds under an additional regularity assumption near $\partial K$. This assumption is used only for the variational formula. The optimization results below require only the finite boundary trace in Definition \ref{mud} and its positive mass.
}
\begin{thm}\label{m1}
Let $K,L\in\mathcal {S}_{o}^n$,  $0<p<\infty$ and let $\mu$ be a Borel measure on $\mathbb{R}^n$ with locally integrable density $g$. {
If $g$ has a representative that is continuous on an open neighborhood of $\partial K$, then
}
	\begin{equation}\label{lmv1}
		\lim_{\varepsilon\rightarrow
			0^+}\dfrac{\mu(K\widetilde{+}_{-p}\varepsilon\cdot
			L)-\mu(K)}{\varepsilon}=-\frac{n}{p} \int_{S^{n-1}}\Big(\dfrac{\rho_L(u)}{\rho_K(u)}\Big)^{-p} d\widetilde{V}_{\mu,K}(u).
	\end{equation}
\end{thm}
\noindent 
In particular, if $\mu$ is the Lebesgue measure on $\mathbb{R}^n$ (i.e.,
$g\equiv 1$ on $\mathbb R^n$), then $\widetilde{V}_{\mu,K}$ reduces to $\widetilde{V}_{K}$. In this case,  \eqref{lmv1} reduces to \eqref{Vp} .

{
Motivated by \eqref{lmv1}, we propose the following optimization problem. Throughout the
optimization problem, $K$ and $\mu$ satisfy Definition \ref{mud}, and the total  mass $|\widetilde{V}_{\mu,K}|=\int_{S^{n-1}}d\widetilde{V}_{\mu,K}(u)$ satisfies
\begin{equation}\label{review-mass}
0<|\widetilde V_{\mu,K}|<\infty.
\end{equation}
Since $\widetilde V_{\mu,K}$ is absolutely continuous with respect to spherical Lebesgue measure, its positive mass implies that its support is not contained in a great subsphere.
}

\noindent\textbf{Problem $\widetilde M_{\mu,p}$.} {\em
	 Given a star
	 {
body $K\in\mathcal {S}_{o}^n$ and $0<p<\infty$, find an origin-symmetric ellipsoid $E$ which solves the following constrained
	minimization problem:
}
\begin{equation*}
	\min_{E\in \mathcal{E}^n_e}V(E)\quad\textrm{subject to}\quad
	\frac{1}{|\widetilde{V}_{\mu,K}|}\int_{S^{n-1}}\Big(\dfrac{\rho_E(u)}{\rho_K(u)}\Big)^{-p} d\widetilde{V}_{\mu,K}(u)\le 1.
\end{equation*} 
{
The endpoint $p=\infty$ is defined by the limiting constraint in \eqref{wuqiong}.
}}
 
The existence and uniqueness of the solution to Problem $\widetilde M_{\mu,p}$  established in this paper will define a new ellipsoid $\widetilde E_{\mu, p}K$. This ellipsoid  can be regarded as the dual $L_p$ John ellipsoid of $K$ for  a  general measure $\mu$, with the following characterization.
\begin{thm}\label{t-2}
	Let $K\in
	\mathcal {S}_{o}^n$  and $0<p\leq\infty$. If 
	{$0<|\widetilde V_{\mu,K}|<\infty$}, then Problem ${\widetilde{M}_{\mu, p}}$ has a unique solution. Moreover, the ellipsoid $\widetilde E_{\mu, p}K$ solves Problem ${\widetilde{M}_{\mu, p}}$ for $0<p<\infty$ if and only if it satisfies 
	\begin{equation}\label{yh1}
		|\widetilde{V}_{\mu,K}|\rho_{\widetilde E_{\mu, p}^{*}K}(x)^{-2}=n\int_{S^{n-1}}|x \cdot u|^2\rho_{\widetilde E_{\mu, p}K}(u)^{2}\Big(\dfrac{\rho_{\widetilde E_{\mu, p}K}(u)}{\rho_K(u)}\Big)^{-p} d\widetilde{V}_{\mu,K}(u),
	\end{equation}
	for all  $x \in \mathbb{R}^n\backslash \{o\}$.
\end{thm}

If $\mu$ is the Lebesgue measure on $\mathbb{R}^n$, then Problem $\widetilde M_{\mu,p}$ becomes  Problem $\widetilde S_{p}$. In this case, $\widetilde E_{\mu, p}K$ is exactly the dual $L_p$ John ellipsoid $\widetilde E_{p}K$,  and \eqref{yh1} reduces to \eqref{s4}.

{
The paper is organized as follows. Section \ref{2} contains the background material and the dual $L_p$ mixed $\mu$-measures. In Section \ref{4}, we prove the existence, uniqueness, and characterization of the solutions to Problem $\widetilde M_{\mu,p}$. Section \ref{5} is devoted to the $L_p$ L\"owner inclusion.
}

	\section{Background materials and  dual $L_p$  mixed $\mu$-measures}\label{2}
In this section, we list some basic facts about convex bodies and star bodies.  The books \cite{Gardner,Schneider} are  good references.
	
	As usual,   let $|x|=\sqrt{x\cdot x}$ be the Euclidean norm of $x$. We denote the Euclidean unit ball by
$
	B^n= \left \{ x\in \mathbb
	R^n:| x | \leq 1 \right \} 
	$
    and the Euclidean unit sphere by
    $
    	S^{n-1}=\left \{ x\in \mathbb
    	R^n:| x | = 1 \right \}.
    $ The volume of $B^n$ is 
    {
denoted by
} $\omega_n$.
    Denote by $C(S^{n-1})$  the set of continuous functions on $S^{n-1}$.

If  $K\in \mathcal {K}^n_{o}$, then the polar body $K^*$ of $K$ is given by
\begin{equation*}
	K^{*}=\{x\in\mathbb R^n: x\cdot y\le 1\quad\textrm{for all}~y\in
	K\}.
\end{equation*}
{
It follows from
} \eqref{JXHS} and \eqref{e-6} that, for $K\in \mathcal {K}^n_{o}$,
 \begin{equation}\label{rh}
	h_{K^*}=1/\rho_{K} \quad \textrm{and} \quad
	\rho_{K^*}=1/h_{K},
\end{equation}
and for $K\in\mathcal {S}_{o}^n$, $c>0$, 
\begin{equation}\label{e-2}
	\rho_{K}(cx)=c^{-1}\rho_K(x), \ \ x\in\mathbb R^n \backslash \left\{o\right\}.
\end{equation}

For $\phi\in \mathrm{GL}(n)$, denote by $\phi K=\{\phi x: x\in K\}$  the image of a set $K\subset\mathbb{R}^n$  under $\phi$. If $K\in \mathcal{K}^n_{o}$, then it is easy to verify that 
\begin{equation}\label{pp}
		(\phi K)^*=\phi^{-t} K^*, 
	\end{equation}
{
where $\phi^{-t}$ denotes the transpose of $\phi^{-1}$
}, and  
\begin{equation}\label{9}
	h_{\phi K}(x)=h_K(\phi^tx),\quad  x\in\mathbb R^n.
\end{equation}
Furthermore, for  $K\in \mathcal{S}^n_{o}$, 
	\begin{equation}\label{rsl}
		\rho_{\phi K}(x)=\rho_K(\phi^{-1}x),\quad   x\in\mathbb
		R^n\backslash\{o\}.
	\end{equation}
In particular, 
\begin{equation}\label{afff}
	\rho_{\phi B^n}(x)=\rho_{B^n}(\phi^{-1} x)=\frac{1}{|\phi^{-1} x|}.
\end{equation}

	\begin{proof}[Proof of Theorem \ref{m1}] 
		Using  polar coordinates, we have
		\begin{align}\label{mm}
			\mu(K)&=\int_K
			g(x)dx\notag\\
			&=\int_{S^{n-1}}\Big(\int_0^{\rho_K(u)}g(ru)r^{n-1}dr\Big)d\sigma_{n-1}(u).\notag
		\end{align}
Then,
\begin{equation}\label{ht1}
\mu(K\widetilde{+}_{-p}\varepsilon\cdot L)=\int_{S^{n-1}} G_{\varepsilon}(u)d\sigma_{n-1}(u),
\end{equation}
where $G_{\varepsilon}(u)=\int_0^{\rho_{\varepsilon}(u)}g(ru)r^{n-1}dr$ and $\rho_{\varepsilon}=\left(\rho_K^{-p}+\varepsilon
\rho_L^{-p}\right)^{-\frac{1}{p}}$.
{
By continuity of $g$ near $r_K(u)$ and the derivative of $\rho_\varepsilon$, we obtain, for $u\in S^{n-1}$,
}
\begin{align}\label{pr}
&\lim_{\varepsilon \to 0^+} \frac{G_{\varepsilon}(u) - G_0(u)}{\varepsilon } \notag\\& = \lim_{\varepsilon \to 0^+} \frac{1}{\varepsilon} \int^{\rho_{\varepsilon}(u)}_{\rho_K(u)} g(ru)r^{n-1} \, dr \notag\\
& = \lim_{\varepsilon \to 0^+} \frac{\rho_{\varepsilon}(u) - \rho_K(u)}{\varepsilon} \cdot \lim_{\varepsilon \to 0^+}\frac{1}{\rho_{\varepsilon}(u) - \rho_K(u)} \int^{\rho_{\varepsilon}(u)}_{\rho_K(u)}  g(ru)r^{n-1} \, dr \notag\\
& =  -\frac{1}{p} \rho_K(u)^{1+p}\rho_L(u)^{-p}g(\rho_K(u)u)\rho_K(u)^{n-1}\notag\\
& =-\frac{1}{p}  \rho_K(u)^{n+p}\rho_L(u)^{-p}g(r_K(u)).
\end{align}
{
There are constants $C,M,R>0$ such that, for all sufficiently small $\varepsilon>0$ and every $u\in S^{n-1}$,
\[
0\leq\rho_K(u)-\rho_\varepsilon(u)\leq C\varepsilon,
\qquad 0<\rho_\varepsilon(u)\leq\rho_K(u)\leq R,
\]
and $g(ru)\leq M$ whenever $\rho_\varepsilon(u)\leq r\leq\rho_K(u)$. Indeed, the radial functions are positive and continuous on the sphere, and the radial intervals lie in a fixed compact neighborhood of $\partial K$ on which $g$ is bounded. It follows that
\[
\left|\frac{G_\varepsilon(u)-G_0(u)}{\varepsilon}\right|
\leq CMR^{n-1}.
\]
This constant is integrable on the sphere. Dominated convergence therefore gives

}\begin{align}
&\lim_{\varepsilon\rightarrow
				0^+}\dfrac{\mu(K\widetilde{+}_{-p}\varepsilon\cdot
				L)-\mu(K)}{\varepsilon}\notag\\
&=\int_{S^{n-1}} G_0'(u) d\sigma_{n-1}(u)  \notag\\
&=-\frac{1}{p}\int_{S^{n-1}}\rho_K(u)^{n+p}\rho_L(u)^{-p}g(r_K(u))d\sigma_{n-1}(u)\notag\\
			&=-\frac{n}{p} \int_{S^{n-1}}\Big(\dfrac{\rho_L(u)}{\rho_K(u)}\Big)^{-p} d\widetilde{V}_{\mu,K}(u), \notag
\end{align}
as desired.
\end{proof}

{
For $K$ and $\mu$ as in Definition \ref{mud}, we define the dual $L_p$ mixed $\mu$-measure by the boundary integral below. When the continuity hypothesis of Theorem \ref{m1} holds, it also has the indicated variational representation:

}
{
\begin{align}\label{LMV}
\widetilde{V}_{\mu,-p}(K,L)&:=\int_{S^{n-1}}\Big(\dfrac{\rho_L(u)}{\rho_K(u)}\Big)^{-p}d\widetilde V_{\mu,K}(u),\notag\\
&=-\frac pn\lim_{\varepsilon\to0^+}\frac{\mu(K\widetilde+_{-p}\varepsilon\cdot L)-\mu(K)}{\varepsilon}.
\end{align}
The second equality holds under the hypothesis of Theorem \ref{m1}.
}
In particular, \begin{equation}\label{KK} \widetilde{V}_{\mu,-p}(K,K)=\int_{S^{n-1}} d\widetilde{V}_{\mu,K}=|\widetilde{V}_{\mu,K}|.\end{equation}
If $\mu$ is the Lebesgue measure on $\mathbb{R}^n$, then $\widetilde{V}_{\mu,-p}(K,L)$ reduces to the dual mixed volume $\widetilde{V}_{-p}(K,L)$ \cite{L96} given by  
\begin{equation*}
	\widetilde{V}_{-p}(K,L)=-\frac{p}{n}\lim_{\varepsilon\rightarrow
		0^+}\dfrac{V(K\widetilde{+}_{-p}\varepsilon\cdot
		L)-V(K)}{\varepsilon}=\int_{S^{n-1}}\Big(\dfrac{\rho_L(u)}{\rho_K(u)}\Big)^{-p} d\widetilde{V}_{K}(u).
\end{equation*}

To facilitate the formulation of Problem  $\widetilde M_{\mu,p}$ for the case $p\rightarrow \infty$, we shall require the following
normalized dual $L_p$  mixed $\mu$-measure $\widehat{V}_{\mu, -p}(K, L)$:
\begin{align}\label{Kp}
	\widehat{V}_{\mu, -p}(K, L)=\left(\frac{\widetilde{V}_{\mu, -p}(K,L)}{\widetilde{V}_{\mu,-p}(K,K)}\right)^{-\frac{1}{p}}=\left(\frac{1}{|\widetilde{V}_{\mu,K}|}
	\int_{S^{n-1}} \Big(\dfrac{\rho_L(u)}{\rho_K(u)}\Big)^{-p} d\widetilde{V}_{\mu,K}(u) \right)^{-\frac{1}{p}},
\end{align}
and 
\begin{align} \label{wuqiong}
\widehat{V}_{\mu, -\infty}(K,L)=\lim_{p\rightarrow \infty} \widehat{V}_{\mu, -p}(K, L)=\min\big\{\rho_L(u)/\rho_K(u):u \in \mathrm{supp}(\widetilde{V}_{\mu,K})\big\}.
\end{align}
Obviously, for $c > 0 $ and $0<p\le \infty$,
\begin{equation} \label{cxx}
\widehat{V}_{\mu, -p}(K,c L)=c\widehat{V}_{\mu, -p}(K, L).
\end{equation}

Let $\mu^{\phi}$ be the measure that has the density function
	$g\circ \phi$ for some $\phi \in \textrm{GL}(n)$. Then,	
		for $\phi\in \textrm{GL}(n)$ and $K\in
		\mathcal {S}_{o}^n$,
	\begin{align}\label{SL}
		\mu(\phi K)=\int_{\phi K} g(x)dx=|\det \phi|\int_{K}g(\phi x)dx=|\det \phi|\mu^\phi( K).
	\end{align}	

The affine nature for the newly  defined dual $L_p$ mixed $\mu$-measure is demonstrated as follows.
\begin{lem}
Let $K,L\in\mathcal {S}_{o}^n$ and $0<p<\infty$. Then, for 
$\phi\in \mathrm{GL}(n)$,
\begin{equation}\label{aff1}
	\widetilde{V}_{\mu, -p}(\phi K, \phi L)=|\det \phi|\widetilde{V}_{\mu^{\phi}, -p}(K,L).
\end{equation}
	\end{lem}
{
\begin{proof}
Put $T_\phi(u)=\phi u/|\phi u|$. The spherical change of variables is
\[
d\sigma_{n-1}(T_\phi u)=|\det\phi|\,|\phi u|^{-n}d\sigma_{n-1}(u).
\]
Moreover,
\[
\rho_{\phi K}(T_\phi u)=|\phi u|\rho_K(u),\qquad
\rho_{\phi L}(T_\phi u)=|\phi u|\rho_L(u).
\]
Lebesgue points and their precise values are preserved under invertible linear changes of variables. Hence Definition \ref{mud} gives
\[
\int_{S^{n-1}}f(v)d\widetilde V_{\mu,\phi K}(v)
=|\det\phi|\int_{S^{n-1}}f(T_\phi u)d\widetilde V_{\mu^\phi,K}(u)
\]
for bounded Borel $f$. Taking $f=(\rho_{\phi L}/\rho_{\phi K})^{-p}$ proves \eqref{aff1} directly from the integral definition, without any additional continuity assumption on $g$.
\end{proof}

}\noindent Therefore, by \eqref{Kp},
\begin{equation} \label{aff6}
	\widehat{V}_{\mu, -p}(\phi K,\phi L)= \widehat{V}_{\mu^\phi, -p}(K, L),
\end{equation}
for any $\phi \in \textrm{GL}(n)$. Taking $p \to \infty $ in \eqref{aff6} yields that
\begin{equation} \label{aff7}
	\widehat{V}_{\mu, -\infty}(\phi K,\phi L)= \widehat{V}_{\mu^\phi, -\infty}(K, L).
\end{equation}

For convenience,  we  also define the dual $L_p$ surface $\mu$-area measure $\widetilde{S}_{\mu,p}(K,\cdot)$ of $K\in\mathcal S_{o}^n$ on $S^{n-1}$  by 
\begin{equation}\label{dcm2}
	\widetilde{S}_{\mu,p}(K, \eta)=\int_\eta\rho_K(u)^{n+p}g(r_K(u))d\sigma_{n-1}(u), \quad 0<p<\infty,
\end{equation}
for any Borel set $\eta\subset S^{n-1}$.
 Then it follows from \eqref{LMV} and \eqref{dcm2} that
\begin{equation}\label{dcm1}
	\widetilde{V}_{\mu,-p}(K,L)=\frac1n\int_{S^{n-1}}\rho_L(u)^{-p} d\widetilde{S}_{\mu,p}(K, u).
\end{equation}
It is worth mentioning that, within the framework of  weighted $L_p$ Brunn-Minkowski theory,  the $L_p$ surface $\mu$-area measure $S_{\mu,
	p}(K, \cdot)$ of  $K \in
\mathcal {K}^n_{o}$    was introduced by Livshyts \cite{Li} for $p=1$ and Wu
\cite{Wu} for $p>1$.

Suppose  $\nu$ is a Borel measure on $S^{n-1}$. For $\phi \in \textrm{GL}(n)$ and $0<p<\infty$, the dual $L_p$ image of $\nu$ under $\phi$, denoted by $\phi_{p^\vdash}\nu$, is defined by
\begin{equation}\label{p1}
	\int_{S^{n-1}} f(u)d\phi_{p^\vdash}\nu(u)=\int_{S^{n-1}}|\phi u|^pf(\overline{\phi u} )d\nu(u), 
\end{equation}
{
for each nonnegative Borel function $f:S^{n-1}\to[0,\infty]$.
} Here $\overline{\phi u}=\phi u/|\phi u|$.

\begin{lem}\label{pp1}
	For $K\in \mathcal {S}^n_{o}$ and $\phi \in \emph{GL}(n)$,
	\begin{equation}\label{aff0}
		d\widetilde{S}_{\mu,p}(\phi K, \cdot) =|\det \phi|d\phi_{p^\vdash} \widetilde{S}_{\mu^{\phi} ,p}(K,\cdot).
	\end{equation}
\end{lem}
\begin{proof}
Using \eqref{aff1}, we  have 
			\begin{equation*} \label{aff2}
\widetilde{V}_{\mu, -p}(\phi K,L)=|\det \phi|   \widetilde{V}_{\mu^{\phi}, -p}(K,\phi^{-1} L).
			\end{equation*}
Equivalently, by \eqref{dcm1},
\begin{equation}\label{r1}
\int_{S^{n-1}}\rho_L(u)^{-p}d\widetilde{S}_{\mu ,p}(\phi K,u)=|\det \phi|\int_{S^{n-1}}\rho_{\phi^{-1}L}(u)^{-p}d\widetilde{S}_{\mu^{\phi} ,p}( K,u).
			\end{equation}
Thus, by \eqref{p1}, \eqref{e-2}, \eqref{rsl},  and \eqref{r1}, we obtain
	\begin{align*}
	\int_{S^{n-1}}\rho_L(u)^{-p} d\phi_{p^\vdash} \widetilde{S}_{\mu^{\phi} ,p}(K,u)
	&=\int_{S^{n-1}}|\phi u|^p\rho_L\big(\overline{\phi u}\big)^{-p} d\widetilde{S}_{\mu^{\phi} ,p}(K,u)
\nonumber \\&=\int_{S^{n-1}}\rho_{\phi^{-1}L}\big( u\big)^{-p} d\widetilde{S}_{\mu^{\phi} ,p}(K,u)
\nonumber \\&=\frac{1}{|\det \phi|}\int_{S^{n-1}}\rho_L(u)^{-p}d\widetilde{S}_{\mu ,p}(\phi K,u).
\end{align*}
{
Since $L$ ranges over all star bodies, the functions $\rho_L^{-p}$ range over all positive continuous functions on $S^{n-1}$. Equality against these functions determines a finite Borel measure. This proves \eqref{aff0}.
}
	\end{proof}

	\section{Dual $L_p$ John ellipsoids for general measures}\label{4}

In view of \eqref{Kp} and \eqref{wuqiong}, Problem $\widetilde{M}_{\mu, p}$ can be reformulated as follows.

\noindent\textbf{
	Problem $\widetilde{M}_{\mu, p}$.} {\em   Given a star
	body $K\in
	\mathcal {S}_{o}^n$ and $0< p\le\infty$, find an origin-symmetric ellipsoid $E$ which solves the following optimization problem:
	\begin{equation}\label{yh}
		\min_{E\in
			\mathcal{E}^n_e}\Big(\frac{V(E)}{\omega_n}\Big)^\frac{1}{n}\quad \textrm{subject~to}
		\quad \widehat{V}_{\mu, -p}(K, E) \ge 1.
\end{equation}}
Its dual problem can be written as:

\noindent\textbf{Problem
	$\widehat{M}_{\mu,p}$.} {\em  Given a star
	body $K\in
	\mathcal {S}_{o}^n$ and $0< p\le\infty$, find an origin-symmetric ellipsoid $E$ which solves the following optimization problem:
\begin{equation}\label{duyh}
	\max_{E\in
		\mathcal{E}^n_e} \widehat{V}_{\mu, -p}(K, E) \quad \textrm{subject~to}
	\quad \Big(\frac{V(E)}{\omega_n}\Big)^\frac{1}{n} \le  1.
\end{equation}}

From \eqref{cxx}, it is easy to see that to find a solution to
\eqref{yh} is equivalent to solve the following problem
\begin{equation}\label{eyh}
	\min_{E\in
			\mathcal{E}^n_e}\Big(\frac{V(E)}{\omega_n}\Big)^\frac{1}{n}\quad \textrm{subject~to}
	\quad \widehat{V}_{\mu, -p}(K, E) = 1.
\end{equation} 
Similarly,  to find a solution to \eqref{duyh} is equivalent  to solve the following problem
\begin{equation}\label{eduyh}
	\max_{E\in
		\mathcal{E}^n_e} \widehat{V}_{\mu, -p}(K, E) \quad \textrm{subject~to}
	\quad \Big(\frac{V(E)}{\omega_n}\Big)^\frac{1}{n}= 1.
\end{equation}

The next lemma shows that the solutions to  Problems ${\widetilde{M}_{\mu, p}}$ and
$\widehat{M}_{\mu,p}$ are closely related.
   
\begin{lem} \label{dual}
Let $K\in  \mathcal {S}_{o}^n$ and $0< p\le\infty$. 
	
$\mathrm{(1)}$ If $E_0$ is a solution to Problem ${\widetilde{M}_{\mu, p}}$, then
$\big(\frac{\omega_n}{V(E_0) }\big)^{{\frac{1}{n}}}E_0$ is a solution  to  Problem $\widehat{M}_{\mu,p}$.

$\mathrm{(2)}$ If $E^{\prime}_0$ is  a solution to Problem  $\widehat{M}_{\mu, p}$, then
$\widehat{V}_{\mu, -p}(K, E^{\prime}_0)^{-1}E^{\prime}_0$ is a solution to Problem ${\widetilde{M}_{\mu,p}}$.
\end{lem}
\begin{proof}
(1) Suppose $E_0$ is a solution to  Problem ${\widetilde{M}_{\mu, p}}$. For any $E\in
\mathcal{E}^n_e$ 
with $(V(E)/\omega_n)^\frac{1}{n}  \le 1$, it follows from \eqref{cxx} that
\begin{equation*}
	\widehat{V}_{\mu, -p}(K, \widehat{V}_{\mu, -p}(K, E)^{-1}E)=1.
\end{equation*}
Thus, the ellipsoid $\widehat{V}_{\mu, -p}(K, E)^{-1}E$ satisfies the constraint condition of  Problem ${\widetilde{M}_{\mu,p}}$. By \eqref{yh} and the fact that $E_0$ is a solution to   Problem ${\widetilde{M}_{\mu, p}}$, we have
\begin{equation}\label{re}
	\Big(\frac{V(E_0)}{\omega_n}\Big)^\frac{1}{n}
	\le\Big(\frac{V(\widehat{V}_{\mu, -p}(K, E)^{-1}E)}{\omega_n}\Big)^\frac{1}{n}
	=\frac{1}{\widehat{V}_{\mu, -p}(K, E)}\Big(\frac{V(E)}{\omega_n}\Big)^\frac{1}{n}.
\end{equation}
Then, by \eqref{cxx}, the fact that $\widehat{V}_{\mu, -p}(K, E_0)=1$, \eqref{re},  and $(V(E)/\omega_n)^\frac{1}{n}  \le 1$,  we  further have
 \begin{align*}
	\widehat{V}_{\mu, -p}\Big(K, \Big(\frac{\omega  _n}{V(E_0)}\Big)^{{\frac{1}{n}}}E_0\Big)
	&=\Big(\frac{\omega  _n}{V(E_0)}\Big)^{{\frac{1}{n}}}\widehat{V}_{\mu, -p}\Big(K, E_0\Big)\nonumber \\
&=\Big(\frac{\omega  _n}{V(E_0)}\Big)^{{\frac{1}{n}}}\nonumber \\
	&\geq \Big(\frac{\omega  _n}{V(E)}\Big)^{{\frac{1}{n}}}\widehat{V}_{\mu, -p}\Big(K, E\Big)\nonumber \\
	&\geq \widehat{V}_{\mu, -p}\Big(K, E\Big).
\end{align*}
Therefore, the ellipsoid $\big(\frac{\omega_n}{V(E_0) }\big)^{{\frac{1}{n}}}E_0$ is a solution to   Problem $\widehat{M}_{\mu, p}$.

(2) Suppose $E^{\prime}_0$ is a solution to  Problem $\widehat{M}_{\mu, p}$. For any $E\in
	\mathcal{E}^n_e$ with $\widehat{V}_{\mu, -p}(K, E) \ge 1$, the ellipsoid $(\omega_n/V(E))^{\frac{1}{n}}E$ satisfies the constraint condition of  Problem ${\widehat{M}_{\mu,p}}$.
By the fact that $E^{\prime}_0$ is a solution to  Problem ${\widehat{M}_{\mu, p}}$,  \eqref{cxx},  and $\widehat{V}_{\mu, -p}(K, E) \ge 1$, we have
\begin{align}\label{rr}
	\widehat{V}_{\mu, -p}(K, E^{\prime}_0)
	&\ge\widehat{V}_{\mu, -p}\Big(K, \Big(\frac{\omega  _n}{V(E)}\Big)^{{\frac{1}{n}}}E\Big)\nonumber \\
	&= \Big(\frac{\omega_n}{V(E)}\Big)^{{\frac{1}{n}}}\widehat{V}_{\mu, -p}(K, E)\nonumber \\
	&\ge \Big(\frac{\omega_n}{V(E)}\Big)^{{\frac{1}{n}}}.
\end{align}
Together with \eqref{rr} and the fact that $V(E^{\prime}_0)=\omega_n$, we further have 
\begin{align*}
	V(\widehat{V}_{\mu, -p}(K, E^{\prime}_0)^{-1}E^{\prime}_0)
	=\widehat{V}_{\mu, -p}(K, E^{\prime}_0)^{-n}V(E^{\prime}_0)\nonumber
	\le\frac{V(E)}{\omega_n}V(E^{\prime}_0)=V(E).
\end{align*}
Therefore, the ellipsoid $\widehat{V}_{\mu, -p}(K, E^{\prime}_0)^{-1}E^{\prime}_0$ is a solution to  Problem ${\widetilde{M}_{\mu, p}}$.
\end{proof}

The existence of the solutions to  Problem $\widetilde{M}_{\mu,p}$ is contained
in the following theorem.

\begin{thm}\label{ph2}Let $K\in \mathcal {S}^n_{o}$ and $0< p< \infty$.  If {$0<|\widetilde V_{\mu,K}|<\infty$}, then
	 there exists a solution to
	Problem $\widetilde{M}_{\mu,p}$.
\end{thm}
\begin{proof} {
By \eqref{dcm2}, the measure $\widetilde S_{\mu,p}(K,\cdot)$ is finite, nonzero, and absolutely continuous with respect to $\sigma_{n-1}$. Thus
}  $\textrm{supp}(\widetilde{S}_{\mu,p}(K, \cdot))$ is  not
	contained in any great subsphere of $S^{n-1}$.  Therefore, there exists a constant $c>0$ such that
	\begin{equation}\label{lower}\min_{v\in S^{n-1}}\int_{S^{n-1}} |v\cdot
		u|^p d\widetilde{S}_{\mu, p}(K, u)>c,\end{equation}
	as the integral in this inequality is positive and continuous on $S^{n-1}$.
	
	From \eqref{eyh}, we shall assume that $\left\{
	E_i\right\}^\infty_{i=1}$ is a minimizing sequence of
	origin-symmetric ellipsoids to Problem $\widetilde{M}_{\mu, p}$ such that
	$\widehat{V}_{\mu, -p}(K, E_i)=1$ and
	\begin{equation}\label{np}
		\lim_{i \rightarrow \infty} V(E_i)=\inf_{E\in \mathcal{E}^n_e}
		\left\{ V(E): \widehat{V}_{\mu, -p}(K, E)=1 \right\}.
	\end{equation}

	Next, we show the sequence of  $E_i^\ast$ is bounded. 	For every $E_i^\ast$,
	there exists a $v_i\in S^{n-1}$ such that
	\begin{equation}\label{e2}
		\frac{1}{2}\textrm{diam}(E_i^\ast)|v_i\cdot u|\leq h_{E_i^\ast}(u),
		\quad\textrm{for~any}~ u\in S^{n-1},
	\end{equation} 
where $\textrm{diam}(E_i^\ast)$ is the diameter of $E_i^\ast$. 
By \eqref{KK}, the assumption of $\widehat{V}_{\mu, -p}(K, E_i)=1$, \eqref{Kp}, \eqref{dcm1}, \eqref{rh}, \eqref{e2}, and \eqref{lower}, we have  
	\begin{align*}
	 |\widetilde{V}_{\mu,K}|&=\widetilde{V}_{\mu,-p}(K,K)=\widetilde{V}_{\mu,-p}(K,E_i)\\
	 &=	\frac1n\int_{S^{n-1}}\rho_{E_i}(u)^{-p} d\widetilde{S}_{\mu,p}(K, u)\\
	 &=\frac1n\int_{S^{n-1}}h_{E_i^\ast}(u)^{p} d\widetilde{S}_{\mu,p}(K, u)\\
	 &\geq \frac1n\Big(\frac{1}{2}\textrm{diam}(E_i^\ast)\Big)^p\int_{S^{n-1}}|v_i\cdot u|^{p} d\widetilde{S}_{\mu,p}(K, u)\\
	 &>\frac cn\Big(\frac{1}{2}\textrm{diam}(E_i^\ast)\Big)^p.
	\end{align*}
That is,  the sequence
	$\left\{ E_i^\ast\right\}^\infty_{i=1}$  is bounded. 
	
	By the Blaschke selection theorem (see, e.g., \cite[Theorem
	1.8.7]{Schneider}), there exists a subsequence of  $\left\{E_i^\ast\right\}$, which will again be denoted by
	$\left\{E_i^\ast\right\}$ such that $\lim_{i \rightarrow \infty} E_i^\ast=L_0$ and $\left\{E_i\right\}$ satisfies \eqref{np}. To see that $L_0$ has non-empty interior, we shall argue by
contradiction. Otherwise, $\lim_{i \rightarrow \infty} V(E_i^\ast)=V(L_0)=0$.
Since $E_i^\ast\in \mathcal{E}^n_e$, we have \begin{equation}\label{fan}
 V(E_i)= \frac{\omega_n^2}{V(E_i^\ast)}\rightarrow\infty,\quad\textrm{if}~i \rightarrow \infty.
	\end{equation}
{
Let $c'=\widehat V_{\mu,-p}(K,B^n)^{-1}$. By \eqref{cxx}, $\widehat V_{\mu,-p}(K,c'B^n)=1$. Thus the infimum in \eqref{np} is at most $V(c'B^n)$, and, for all sufficiently large $i$,
\[
V(E_i)\leq V(c'B^n)+1.
\]
This contradicts \eqref{fan}. Therefore $L_0$ has nonempty interior. To see also that it is an ellipsoid, write $h_{E_i^*}(u)^2=u\cdot A_i u$, where the symmetric positive-definite matrices $A_i$ are uniformly bounded. A subsequence converges to a positive-semidefinite matrix $A$, and $h_{L_0}(u)^2=u\cdot Au$. Nonempty interior makes $A$ positive definite. Hence $L_0$ is an origin-symmetric ellipsoid. Polarity is continuous at $L_0$, so $E_i\to L_0^*$. The integral constraint passes to the limit by uniform convergence of $\rho_{E_i}^{-p}=h_{E_i^*}^p$ and finiteness of $\widetilde S_{\mu,p}(K,\cdot)$. Consequently $L_0^*$ is a solution to Problem $\widetilde M_{\mu,p}$.

}\end{proof}

\begin{thm}\label{LL}
Let $K\in  \mathcal {S}_{o}^n$ and  $0< p<\infty$.   If  $B^n$ solves  Problem ${\widehat{M}_{\mu, p}}$, then it is the unique solution. Moreover, $B^n$ solves Problem $\widehat{M}_{\mu, p}$  if and only if, for  $x \in \mathbb{R}^n  $,
	\begin{equation}\label{wy}
		\widetilde{V}_{\mu, -p}(K,B^n)|x|^2=\int_{S^{n-1}} \left | x\cdot u \right |^2d\widetilde{S}_{\mu,p}(K, u).
	\end{equation}
\end{thm}

\begin{proof}
{
Let $T$ be a real $n\times n$ matrix.
} Choose $\varepsilon_0$ sufficiently small so that $I+\varepsilon T$ is invertible for all $\varepsilon\in(-\varepsilon_0,\varepsilon_0)$. Let \begin{equation*}
		T_\varepsilon=\frac{I+\varepsilon T}{\mathrm{det}(I+\varepsilon T)^\frac{1}{n}},
	\end{equation*}
for $\varepsilon\in(-\varepsilon_0,\varepsilon_0)$. 
Denote by $E_\varepsilon =T_\varepsilon^{-1} B^n$. Clearly, $V(E_\varepsilon)=\omega_n$. If $B^n$ is an ${\widehat{M}_{\mu, p}}$ solution, then it follows from \eqref{eduyh} and \eqref{Kp} that $\widetilde{V}_{\mu, -p}(K,B^n)\leq \widetilde{V}_{\mu, -p}(K,E_\varepsilon)$. {
The integrand and its derivative are uniformly bounded for small $\varepsilon$, and $\widetilde S_{\mu,p}(K,\cdot)$ is finite. Hence differentiation under the integral is justified, and \eqref{dcm1} gives
} 
	\begin{equation*}
		\frac{d}{d\varepsilon}\Big(\widetilde{V}_{\mu, -p}(K,E_\varepsilon)\Big)\bigg|_{\varepsilon = 0}=\frac{d}{d\varepsilon} \Big(\frac{1}{n}\int_{S^{n-1}}\rho_{E_\varepsilon}(u)^{-p}d\widetilde{S}_{\mu,p}(K, u)\Big)\bigg|_{\varepsilon = 0}=0.
	\end{equation*}
Note that, by \eqref{afff}, 
\begin{equation*}
	\rho_{E_\varepsilon}(u)=\rho_{T_\varepsilon^{-1} B^n}(u)=\frac{1}{|T_\varepsilon u|}=\mathrm{det}(I+\varepsilon T)^\frac{1}{n}(u \cdot u+2\varepsilon u \cdot Tu+\varepsilon^2Tu\cdot Tu)^{-\frac{1}{2}}
	\end{equation*}
and 
\begin{equation*}
\frac{d}{d\varepsilon}\Big(\mathrm{det}(I+\varepsilon T)\Big)\bigg|_{\varepsilon = 0}=\mathrm{tr}(T).
	\end{equation*}
Then we have 
		\begin{equation}\label{11}
		\widetilde{V}_{\mu, -p}(K, B^n)\mathrm{tr}(T)=\int_{S^{n-1}}  (u\cdot Tu)d\widetilde{S}_{\mu,p}(K, u),
	\end{equation}
where 
\begin{equation*}
		\widetilde{V}_{\mu, -p}(K, B^n)=\frac1n\int_{S^{n-1}}  d\widetilde{S}_{\mu,p}(K, u).
	\end{equation*}
Let $e_1,\ldots,e_n$ denote the canonical basis for $\mathbb R^n$. Choosing  $T=e_i\otimes e_j$ for  $i,j\in\left \{ 1,...,n \right \}$ in \eqref{11}, we get
	\begin{equation}\label{v1}
		\widetilde{V}_{\mu, -p}(K, B^n)\delta_{i,j}=\int_{S^{n-1}} (u\cdot e_i)(u\cdot e_j)d\widetilde{S}_{\mu,p}(K, u),
	\end{equation}
which is equivalent to the desired equation \eqref{wy}.

{
Conversely, suppose \eqref{wy} holds. Put
\[
m=\widetilde V_{\mu,-p}(K,B^n),\qquad
d\nu(u)=\frac{1}{nm}\,d\widetilde S_{\mu,p}(K,u).
\]
Then $\nu$ is a probability measure, and \eqref{wy} is equivalent to
\[
\int_{S^{n-1}}(x\cdot u)^2d\nu(u)=\frac{|x|^2}{n}.
\]
Every origin-symmetric ellipsoid of volume $\omega_n$ can be written as $PB^n$, with $P$ symmetric positive definite and $\det P=1$. Write $P^{-1}=Q^tDQ$, where $Q$ is orthogonal and $D=\operatorname{diag}(a_1,\ldots,a_n)$, with $a_i>0$ and $\prod_i a_i=1$. Let $\nu_Q$ be the image of $\nu$ under $u\mapsto Qu$. This measure is again isotropic, with $\int u_i^2d\nu_Q=1/n$. By concavity of the logarithm,
\begin{align*}
\int_{S^{n-1}}\log|P^{-1}u|\,d\nu(u)
&=\frac12\int_{S^{n-1}}\log\Big(\sum_{i=1}^n a_i^2u_i^2\Big)d\nu_Q(u)\\
&\geq\sum_{i=1}^n\log a_i\int_{S^{n-1}}u_i^2d\nu_Q(u)
=\frac1n\sum_{i=1}^n\log a_i=0.
\end{align*}
Jensen's inequality for the exponential now gives, for every $p>0$,
\[
\frac{\widetilde V_{\mu,-p}(K,PB^n)}{m}
=\int_{S^{n-1}}|P^{-1}u|^p d\nu(u)
\geq\exp\left(p\int_{S^{n-1}}\log|P^{-1}u|d\nu(u)\right)\geq1.
\]
If the $a_i$ are not all equal, the logarithmic inequality is strict outside a finite union of proper great subspheres. Since $\nu_Q$ is absolutely continuous and has mass one, the integrated inequality is then strict. Equality therefore requires $a_1=\cdots=a_n=1$, or $P=I$. Conversely, $P=I$ gives equality. A solution to Problem $\widehat M_{\mu,p}$ must have volume $\omega_n$ by homogeneity, so this also proves uniqueness among all feasible ellipsoids. This argument follows the logarithmic method of Lutwak, Yang, and Zhang \cite[proof of Theorem 2.2]{LYZJ} and applies throughout $p>0$.
}\end{proof}

We are now in position to characterize  the solutions to Problem $\widehat{M}_{\mu,p}$.

\begin{thm}\label{t-1}
Let $K\in  \mathcal {S}_{o}^n$ and $0< p<\infty$. If  {$0<|\widetilde V_{\mu,K}|<\infty$}, then  {
Problem ${\widehat{M}_{\mu, p}}$ has a unique solution. Moreover, an origin-symmetric ellipsoid $E$ solves Problem ${\widehat{M}_{\mu, p}}$ if and only if $V(E)=\omega_n$ and it satisfies
} 
 \begin{equation}\label{yh11}
 \widetilde{V}_{\mu, -p}(K,E)\rho_{E^{*}}(x)^{-2}=\int_{S^{n-1}} |x \cdot u|^2 \rho_E(u)^{2-p}d\widetilde{S}_{\mu,p}(K, u),
 \end{equation}
for all  $x \in \mathbb{R}^n\backslash \{o\}$.
\end{thm}
\begin{proof} 	According to Lemma \ref{dual} and Theorem \ref{ph2}, we only need to prove \eqref{yh11} by using Theorem \ref{LL}. 
Assume that $E=\phi B^n$ solves   Problem $\widehat{M}_{\mu,p}$ for some $\phi\in \textrm{SL}(n)$. Then it follows from \eqref{aff6} that
\begin{equation*}
	\widehat{V}_{\mu, -p}(K, E)=\widehat{V}_{\mu, -p}(K,
	\phi B^n)=\widehat{V}_{\mu^{\phi}, -p}(\phi^{-1}K, B^n).
\end{equation*}
By \eqref{eduyh}, it means that the unit ball $B^n$ solves  Problem
 $\widehat{M}_{\mu^{\phi}, p}$ for $\phi^{-1}K$. Thus, by Theorem \ref{LL}, \eqref{aff0}, \eqref{p1}, \eqref{e-2} and \eqref{rsl}, we have 

\begin{align}\label{pf3}
	\widetilde{V}_{\mu^{\phi} , -p}(\phi^{-1}K,B^n)|x |^2 \nonumber 
&=\int_{S^{n-1}}|x \cdot u |^2d\widetilde{S}_{\mu^\phi,p}(\phi^{-1} K, u)\nonumber \\
	&= |\det \phi^{-1}|\int_{S^{n-1}}|x \cdot u |^2 d\phi_{p^\vdash}^{-1} \widetilde{S}_{\mu ,p}(K,u)\nonumber \\
	&= |\det \phi^{-1}|\int_{S^{n-1}}|x \cdot \overline{\phi^{-1} u} |^2|\phi^{-1}u|^pd \widetilde{S}_{\mu ,p}(K,u)\nonumber \\
		&= |\det \phi^{-1}|\int_{S^{n-1}}|x \cdot \phi^{-1} u|^2|\phi^{-1}u|^{p-2}d \widetilde{S}_{\mu ,p}(K,u).
\end{align}
But, by \eqref{aff1},
\begin{equation*}
	\widetilde{V}_{\mu^{\phi}, -p}\left(\phi^{-1}K, B^n\right)=|\det
	\phi^{-1}|\widetilde{V}_{\mu, -p}\left(K, \phi B^n\right).
\end{equation*}
Therefore, if we take  $z=\phi^{-t}x$ in \eqref{pf3}, then
\begin{align}\label{xz}
    \widetilde{V}_{\mu, -p}\left(K,
	\phi B^n\right)|\phi^t z|^2
	&=\int_{S^{n-1}}|\phi^t z\cdot \phi^{-1} u|^2
	|\phi^{-1} u |^{p-2} d \widetilde{S}_{\mu ,p}(K,u)\nonumber \\
	&=\int_{S^{n-1}}|z\cdot  u|^2
	|\phi^{-1} u |^{p-2} d \widetilde{S}_{\mu ,p}(K,u).
\end{align}
For $E=\phi B^n$ and $z \in \mathbb{R}^n\backslash \{o\}$, it follows from \eqref{pp} and \eqref{afff} that
\begin{equation*}
	\rho_{E}(z)=\rho_{\phi B^n}(z)=|\phi^{-1} z|^{-1}\quad \textrm{and}\quad \rho_{E^{*}}(z)=\rho_{\phi^{-t} B^n}(z)=|\phi^t z|^{-1}.
\end{equation*}
Consequently,  \eqref{xz} becomes
\begin{align*}
\widetilde{V}_{\mu, -p}\left(K,
E\right)\rho_{E^{*}}(z)^{-2}
	&=\int_{S^{n-1}}| z\cdot u|^2\rho_E(u) ^{2-p}
	d \widetilde{S}_{\mu ,p}(K,u),
\end{align*}
{
as desired. Conversely, suppose that $V(E)=\omega_n$ and \eqref{yh11} holds. Write $E=\phi B^n$ with $\det\phi=1$. All the changes of variables above are reversible. Hence \eqref{wy} holds for $(\mu^\phi,\phi^{-1}K)$. Theorem \ref{LL} shows that $B^n$ is the unique normalized solution for this transformed pair. By \eqref{aff6}, $E$ is the unique solution for $(\mu,K)$.
}
\end{proof}

{
For $0<p<\infty$, Theorem \ref{t-2} now follows from Lemma \ref{dual} and Theorem \ref{t-1}. For completeness, if $E$ satisfies \eqref{yh1}, write $E=\phi B^n$, substitute $x=\phi^{-t}e_i$, and sum over $i$. Since $\sum_i|\phi^{-t}e_i\cdot u|^2=\rho_E(u)^{-2}$, this gives $\widetilde V_{\mu,-p}(K,E)=|\widetilde V_{\mu,K}|$. Thus the constraint is satisfied, and \eqref{yh1} becomes \eqref{yh11}. Normalizing the volume and applying Theorem \ref{t-1} and Lemma \ref{dual} proves sufficiency. The endpoint $p=\infty$ is treated next.
}
 
In view of \eqref{wuqiong} and \eqref{yh},  Problem $\widetilde{M}_{\mu, \infty}$ can be rephrased  as follows. 

\noindent\textbf{Problem $\widetilde{M}_{\mu, \infty}$.} {\em    Given a star
		body $K\in
		\mathcal {S}_{o}^n$, find an origin-symmetric ellipsoid $E$ which solves the following constrained minimization problem:
	\begin{equation*}
		\min_{E\in
			\mathcal{E}^n_e} V(E)\quad \textrm{subject~to}
		\quad \widetilde{K}\subseteq E,
\end{equation*}
where $\widetilde{K}$ is a convex hull in $\mathbb{R}^n$ defined by \begin{equation}\label{e-9}
	\widetilde{K}=\mathrm{conv}\big\{r_K(u): u \in \mathrm{supp}(\widetilde{V}_{\mu,K})\big\}.
\end{equation}}
{
Since $\widetilde K$ is the convex hull of a compact subset of $\partial K$, it is compact and contained in $\operatorname{conv}K$. Put
\[
C=\operatorname{conv}\{\widetilde K\cup(-\widetilde K)\}.
\]
The support of $\widetilde V_{\mu,K}$ spans $\mathbb R^n$, so $C$ is an origin-symmetric convex body with nonempty interior. An origin-symmetric ellipsoid contains $\widetilde K$ if and only if it contains $C$. The unique L\"owner ellipsoid of $C$ is origin-symmetric, by uniqueness and the symmetry of $C$, and is therefore precisely the unique solution to Problem $\widetilde M_{\mu,\infty}$. This completes the proof of Theorem \ref{t-2}. The survey \cite{Henk} contains further background on L\"owner ellipsoids.
}

Therefore, for $K\in
\mathcal {S}_{o}^n$ and $0< p\le\infty$, we call the
solution to  Problem $\widetilde{M}_{\mu, p}$ the dual $L_p$ John ellipsoid of $K$ for the measure $\mu$   and denote it by
$\widetilde{E}_{\mu, p}K$. The solution to   Problem $\widehat{M}_{\mu,p}$  is called the normalized dual $L_p$ John ellipsoid of $K$ for
the measure $\mu$   and  denoted by $\widehat{E}_{\mu, p}K$.

{
The following monotonicity inequality is the weighted counterpart of \cite[Theorem 5.2]{YLW}.
} 
\begin{thm}\label{vv1}
Let $K\in
	\mathcal {S}_{o}^n$ and $0< p\leq q\leq \infty$. If  {$0<|\widetilde V_{\mu,K}|<\infty$}, then
	\begin{equation*}
		V(\widetilde{E}_{\mu, p}K)\leq V(\widetilde{E}_{\mu, q}K).
	\end{equation*}
\end{thm}
{
\begin{proof}
Apply monotonicity of probability $L_p$ means to $\rho_K/\rho_E$ with respect to $d\widetilde V_{\mu,K}/|\widetilde V_{\mu,K}|$. For $0<p\leq q\leq\infty$, this gives $\widehat V_{\mu,-p}(K,E)\geq\widehat V_{\mu,-q}(K,E)$. In particular, $\widetilde E_{\mu,q}K$ satisfies the constraint
in Problem $\widetilde M_{\mu,p}$. By the minimizing property
of $\widetilde E_{\mu,p}K$, we obtain
\[
V(\widetilde E_{\mu,p}K)
\leq V(\widetilde E_{\mu,q}K).
\]
The same argument applies when $q=\infty$.
\end{proof}}

It was shown in
  \cite[Theorem 5.4]{YLW} that {
the volume of the dual $L_p$ John ellipsoid $\widetilde{E}_pK$ is at least that of $K$.
} Imposing a homogeneous condition on the density $g$, we get the following theorem. 

\begin{thm}
Let $K\in
\mathcal {S}_{o}^n$ and $0<p<\infty$. If  {$0<|\widetilde V_{\mu,K}|<\infty$} {
and $g$ is $s$-homogeneous on $\mathbb R^n\setminus\{o\}$ with $s>-n$, then
}
	\begin{equation}\label{vol1}
		\mu(\widetilde{E}_{\mu, p}K)\geq \mu(K),
	\end{equation}
	{
with equality if and only if $\rho_K=\rho_{\widetilde E_{\mu,p}K}$ on $\operatorname{supp}(\widetilde V_{\mu,K})$. In particular, if this support is $S^{n-1}$, equality holds if and only if $K=\widetilde E_{\mu,p}K$.
}
\end{thm}
\begin{proof}
We first claim \begin{equation}\label{cl}
	|\widetilde{V}_{\mu,K}|=\frac{n+s}{n}\mu(K),
\end{equation}
if $g$ is $s$-homogeneous. In fact, using polar coordinates gives
	\begin{align}\label{my1}
			\mu(K)&=\int_K
		g(x)dx\notag\\
		&=\int_{S^{n-1}}\Big(\int_0^{\rho_K(u)}r^{n-1}g(ru)dr\Big)d\sigma_{n-1}(u)\notag\\
		&=\int_{S^{n-1}}\Big(\int_0^{\rho_K(u)}r^{n+s-1}g(u)dr\Big)d\sigma_{n-1}(u)\notag\\
		&=\frac{1}{n+s}\int_{S^{n-1}}\rho_K(u)^{n+s}g(u)d\sigma_{n-1}(u).
	\end{align}
On the other hand, it follows from \eqref{KK} and \eqref{e-7} that
\begin{align}\label{my2}
	|\widetilde{V}_{\mu,K}|
		&=\int_{S^{n-1}} d\widetilde{V}_{\mu,K}\notag\\
	&=\frac{1}{n}\int_{S^{n-1}}\rho_K(u)^{n}g(\rho_K(u)u)d\sigma_{n-1}(u)\notag\\
	&=\frac{1}{n}\int_{S^{n-1}}\rho_K(u)^{n+s}g(u)d\sigma_{n-1}(u).
\end{align}
Therefore, the claim \eqref{cl} immediately follows from \eqref{my1} and \eqref{my2}.

Moreover, it follows from \eqref{eyh} that $\widehat{V}_{\mu, -p}(K, \widetilde{E}_{\mu, p}K)= 1$. Combining this with  \eqref{Kp}, \eqref{KK}, and \eqref{cl}, we obtain
\begin{equation}\label{mp2}
	\widetilde{V}_{\mu, -p}(K, \widetilde{E}_{\mu, p}K)= \widetilde{V}_{\mu,-p}(K,K)=|\widetilde{V}_{\mu,K}|=\frac{n+s}{n}\mu(K).
\end{equation}
By \eqref{LMV}, the H\"{o}lder inequality, and \eqref{my1}, we have
\begin{align*}
&\widetilde{V}_{\mu,-p}(K,\widetilde{E}_{\mu, p}K)\nonumber \\
	&=\frac{1}{n}\int_{S^{n-1}}\rho_K(u)^{n+p}\rho_{\widetilde{E}_{\mu, p}K}(u)^{-p}g(\rho_K(u)u)d\sigma_{n-1}(u) \nonumber \\
	&=\frac{n+s}{n}\cdot\frac{1}{n+s}\int_{S^{n-1}}\rho_K(u)^{n+p+s}\rho_{\widetilde{E}_{\mu, p}K}(u)^{-p}g(u)d\sigma_{n-1}(u) \nonumber \\
	&{\geq \frac{n+s}{n} \left(\frac{1}{n+s}\int_{S^{n-1}}\rho_K(u)^{n+s}g(u)d\sigma_{n-1}(u)\right)^{\frac{n+p+s}{n+s}}}\nonumber \\
	&\qquad {\cdot \left(\frac{1}{n+s}\int_{S^{n-1}}\rho_{\widetilde{E}_{\mu, p}K}(u)^{n+s}g(u)d\sigma_{n-1}(u)\right)^{\frac{-p}{n+s}}}\nonumber \\
	&=\frac{n+s}{n}	\mu(K)^{\frac{n+p+s}{n+s}} \mu(\widetilde{E}_{\mu, p}K)^{\frac{-p}{n+s}}.
\end{align*}
Together with \eqref{mp2}, we further have 
\begin{equation*}
	\frac{n+s}{n}\mu(K)\geq \frac{n+s}{n}	\mu(K)^{\frac{n+p+s}{n+s}} \mu(\widetilde{E}_{\mu, p}K)^{\frac{-p}{n+s}}.
\end{equation*}
{
This proves \eqref{vol1}. Equality in H\"older's inequality holds precisely when $\rho_K/\rho_{\widetilde E_{\mu,p}K}$ is constant almost everywhere with respect to $g(u)d\sigma_{n-1}(u)$. The constraint \eqref{mp2} forces this constant to be one. Since both radial functions are continuous, equality almost everywhere extends to the support of this measure, which, by homogeneity and positivity of $\rho_K$, is $\operatorname{supp}(\widetilde V_{\mu,K})$. The converse follows from the same integral formulas. Full spherical support then gives equality of the two bodies.
} 
\end{proof}

\begin{rem}
An important special case of the measure $\mu$ is the Gaussian measure $\gamma_n$ on $\mathbb{R}^n$, defined by the density $g(x)=\frac{1}{(\sqrt{2\pi } )^n}e^{-|x|^2/2}$. In particular, \eqref{LMV}  leads to the following variational formula:
\begin{align*}\label{gau}
	\widetilde{V}_{\gamma_n,-p}(K,L)
	&=-\frac pn\lim_{\varepsilon\rightarrow
		0^+}\frac{\gamma_n(K\widetilde{+}_{-p}\varepsilon\cdot
		L)-\gamma_n(K)}{\varepsilon}\nonumber\\
	&=\frac{1}{n(\sqrt{2\pi } )^n}\int_{S^{n-1}}\rho_K(u)^{n+p}\rho_L(u)^{-p}e^{-\rho_K^2(u)/2}d\sigma_{n-1}(u).
\end{align*}
Note that the corresponding variational formula for the Minkowski–Firey $L_p$-combination \eqref{lp} was established by Huang, Xi, and Zhao \cite{YXZ} for $p=1$ and by Liu \cite{Lj} for $p \geq 1$.  The ellipsoid $\widetilde E_{\gamma_n, p}K$, which solves Problem $\widetilde{M}_{\gamma_n, p}$,
is called the dual $L_p$ John ellipsoid of $K$ for the Gaussian measure $\gamma_n$. Furthermore, Theorem \ref{t-2} implies that  $\widetilde E_{\gamma_n, p}K$ is a solution to Problem ${\widetilde{M}_{\gamma_n, p}}$ for $0<p<\infty$ if and only if it satisfies 
\begin{equation*}\label{yh2}
	|\widetilde{V}_{\gamma_n,K}|{
\rho_{\widetilde E_{\gamma_n, p}^{*}K}(x)^{-2}=n\int_{S^{n-1}}
}|x \cdot u|^2\rho_{\widetilde E_{\gamma_n, p}K}(u)^{2}\Big(\dfrac{\rho_{\widetilde E_{\gamma_n, p}K}(u)}{\rho_K(u)}\Big)^{-p} d\widetilde{V}_{\gamma_n,K}(u),
\end{equation*}
for all $x \in \mathbb{R}^n\backslash \{o\}$.	
	\end{rem}

\section{The $L_p$ L\"{o}wner inclusion}\label{5}

The affine nature for $\widetilde E_{\mu, p} K$ is as follows. 

\begin{lem}\label{fs}
If $K\in
\mathcal {S}_{o}^n$ and  $0<p\le\infty$, then for $\phi \in
	\emph{GL}(n)$,
	\begin{equation}\label{Ea}
		\widetilde E_{\mu, p}\phi K=\phi  \widetilde E_{\mu^{\phi}, p}K.
	\end{equation}
\end{lem}

\begin{proof}
		By  \eqref{aff6} and \eqref{aff7}, we have, for $\phi\in \textrm{GL}(n)$,
	\begin{equation*}
		1=\widehat{V}_{\mu, -p}(\phi K,\widetilde E_{\mu, p}\phi
		K)=\widehat{V}_{\mu^{\phi}, -p}(K, \phi^{-1}\widetilde E_{\mu, p}\phi K).
	\end{equation*}
	Then, it follows from \eqref{eyh} that
	\begin{equation}\label{lee}
		V(\phi^{-1}\widetilde E_{\mu, p}\phi K)\ge V( \widetilde E_{\mu^{\phi}, p}K).
	\end{equation}
	Using \eqref{aff6} and \eqref{aff7} again, we also get
	\begin{equation*}\label{eq}
		1=\widehat{V}_{\mu^{\phi}, -p}(K, \widetilde E_{\mu^{\phi}, p}
		K)=\widehat{V}_{\mu, -p}(\phi K, \phi \widetilde E_{\mu^{\phi}, p} K),
	\end{equation*}
and by \eqref{eyh} again,
	\begin{equation}\label{ee2}
		V(\phi \widetilde E_{\mu^{\phi}, p}K)\ge V(\widetilde E_{\mu, p}\phi K).
	\end{equation}
	Combining \eqref{lee} with \eqref{ee2} yields
	\begin{equation*}\label{yi}
		V(\phi \widetilde E_{\mu^{\phi}, p}K)= V(\widetilde E_{\mu, p}\phi K).
	\end{equation*}
Now, the uniqueness of the dual $L_p$ John ellipsoid for the
measure $\mu$ immediately yields {
the desired result
} \eqref{Ea}.	
\end{proof}

{
For $K\in
\mathcal {S}_{o}^n$ and $1\leq p<\infty$, the weighted $L_p$ centroid body
}  $\Gamma_{\mu,p}K$ is a convex body in $\mathbb{R}^n$ whose support function at $x\in \mathbb R^n$ is given by
\begin{align}\label{as}
	h_{\Gamma_{\mu,p}K}(x)^p&=\frac{1}{|\widetilde{V}_{\mu,K}|}\int_{S^{n-1}} |x\cdot u|^{p} \rho_K(u)^{n+p}g(\rho_K(u)u)d\sigma_{n-1}(u)\notag\\
	&=\frac{1}{|\widetilde{V}_{\mu,K}|}\int_{S^{n-1}} |x\cdot u|^{p} d \widetilde{S}_{\mu ,p}(K,u),
\end{align}
{Minkowski's integral inequality shows that the positive homogeneous function defined in \eqref{as} is sublinear. Its value is positive for every nonzero $x$, since $\widetilde S_{\mu,p}(K,\cdot)$ is nonzero and absolutely continuous. It therefore defines an origin-symmetric convex body with nonempty interior.}
Equivalently, by \eqref{e-7},
\begin{equation*}
	h_{\Gamma_{\mu,p}K}(x)^p=\frac{n}{|\widetilde{V}_{\mu,K}|}\int_{S^{n-1}} \Big(|x \cdot u|\rho_{K}(u)\Big)^{p} d\widetilde{V}_{\mu,K}(u),
\end{equation*}
and for $p=\infty$,
\begin{align}\label{e-11}
h_{\Gamma_{\mu,\infty}K}(x)\nonumber
&=\lim_{p \to \infty}h_{\Gamma_{\mu,p}K}(x)\nonumber\\
&=\lim_{p \to \infty}\left(\frac{n}{|\widetilde{V}_{\mu,K}|}\int_{S^{n-1}} \Big(|x \cdot u|\rho_{K}(u)\Big)^{p} d\widetilde{V}_{\mu,K}(u)\right)^\frac{1}{p}\\
&=\max_{u\in \mathrm{supp}(\widetilde{V}_{\mu,K})}( |x \cdot u|\rho_{K}(u)).\nonumber
\end{align}
The definition above immediately implies that 
\begin{equation}\label{e-10}
\Gamma_{\mu,\infty}K=\mathrm{conv}\big\{\widetilde{K}\cup (-\widetilde{K})\big\},
\end{equation}
where $\widetilde{K}$ is defined in \eqref{e-9}. 
 Moreover, it follows from \eqref{yh1} and \eqref{rh} that 
$$\widetilde{E}_{\mu, 2}K=\Gamma_{\mu, 2} K.$$
{
When $\mu$ is Lebesgue measure, polar coordinates give
\[
h_{\Gamma_{\mu,p}K}(x)^p=\frac{n+p}{V(K)}\int_K|x\cdot y|^pdy.
\]
Thus this body is a scalar multiple of the classical $L_p$ centroid body of \cite{LZ97}, whose normalization fixes $B^n$. More precisely, if $c_{n,p}=V(B^n)^{-1}\int_{B^n}|y_1|^pdy$, then
\[
\Gamma_{\mu,p}K=\big((n+p)c_{n,p}\big)^{1/p}\Gamma_pK.
\]
The two normalizations agree when $p=2$.
} 


\begin{lem}
Let $K\in \mathcal {S}^n_{o}$ and $1\leq p\leq \infty$. Then for $\phi \in
\emph{GL}(n)$,
\begin{equation}\label{zxt}
	\Gamma_{\mu,p}\phi K=\phi \Gamma_{\mu^\phi,p}K.
\end{equation}
\end{lem}
\begin{proof}
	By \eqref{as}, \eqref{KK}, \eqref{aff1}, \eqref{aff0}, \eqref{p1}, and \eqref{9}, we have
	\begin{align}\label{pw}
		h_{\Gamma_{\mu,p}\phi K}(x)
				&=\left(\frac{1}{|\widetilde{V}_{\mu,\phi K}|}\int_{S^{n-1}}\left | x \cdot u\right |^{p}d \widetilde{S}_{\mu ,p}(\phi K,u)\right)^\frac{1}{p}\nonumber\\
			&=\left(\dfrac{1}{ |\widetilde{V}_{\mu^\phi,K}|}\int_{S^{n-1}}\left |x \cdot u\right |^{p} d\phi_{p^\vdash} \widetilde{S}_{\mu^{\phi} , p}(K,u)\right)^\frac{1}{p}\nonumber\\
		&=\left(\dfrac{1}{|\widetilde{V}_{\mu^\phi ,K}|}\int_{S^{n-1}}\left | x \cdot \overline{\phi u}\right |^{p} |\phi u|^p d \widetilde{S}_{\mu^{\phi} , p}(K,u)\right)^\frac{1}{p}\nonumber\\
		&=\left(\dfrac{1}{|\widetilde{V}_{\mu^\phi ,K}|}\int_{S^{n-1}}\left | \phi^t x\cdot u\right |^{p}d \widetilde{S}_{\mu^{\phi} , p}(K,u)\right)^\frac{1}{p}\nonumber\\
		&=	h_{\Gamma_{\mu^\phi ,p}K}(\phi^t x)= h_{\phi \Gamma_{\mu^\phi ,p}K}(x),
	\end{align}
for any $x\in \mathbb R^n$ and $\phi \in
\mathrm{GL}(n)$. 

For $p=\infty$, by \eqref{e-11} and \eqref{pw}, we also have
\begin{align*}
h_{\Gamma_{\mu,\infty}\phi K}(x)&=\lim_{p \to \infty}\left(\frac{n}{|\widetilde{V}_{\mu,\phi K}|}\int_{S^{n-1}} \Big(|x \cdot u|\rho_{\phi K}(u)\Big)^{p} d\widetilde{V}_{\mu,\phi K}(u)\right)^\frac{1}{p}\\
&=\lim_{p \to \infty}\left(\frac{n}{|\widetilde{V}_{\mu^\phi,K}|}\int_{S^{n-1}} \Big(|\phi^tx \cdot u|\rho_{K}(u)\Big)^{p} d\widetilde{V}_{\mu^\phi,K}(u)\right)^\frac{1}{p}\\
&=h_{\Gamma_{\mu^\phi,\infty} K}(\phi^tx)\\
&=h_{\phi\Gamma_{\mu^\phi,\infty} K}(x),	
\end{align*}
for any $x\in \mathbb R^n$ and $\phi \in
\mathrm{GL}(n)$. 
	\end{proof}

As the dual of the John inclusion, the L\"{o}wner inclusion  asserts  that if $K$ is an origin-symmetric convex
body in $\mathbb R^n$, then
\begin{equation}\label{LK}
	\frac{1}{\sqrt{n}}LK\subseteq K\subseteq LK
\end{equation}
where  $LK$ denotes  the  L\"{o}wner ellipsoid of $K$. Finally, we  establish the $L_p$ L\"{o}wner inclusion for the new dual $L_p$ John ellipsoid.

\begin{lem}\label{tl}
If $1\leq p\leq \infty$ and $K\in \mathcal {S}^n_{o}$ such that $\widetilde{E}_{\mu, p}K=B^n$, then  
	$$ \widetilde{E}_{\mu, p}K\subseteq \Gamma_{\mu,p}K\subseteq n^{\frac{1}{p}-\frac{1}{2}}\widetilde{E}_{\mu, p}K, \quad 1\leq p\leq 2;$$	$$n^{\frac{1}{p}-\frac{1}{2}}\widetilde{E}_{\mu, p}K \subseteq	\Gamma_{\mu,p}K \subseteq\widetilde{E}_{\mu, p}K ,\quad  2\leq p\leq \infty.$$
	\end{lem}
\begin{proof}
When $1\leq p\leq 2$, by the assumption of $\widetilde{E}_{\mu, p}K=B^n$, \eqref{yh1} and \eqref{dcm2}, we have 
\begin{equation}\label{bn}
	|\widetilde{V}_{\mu,K}|=\int_{S^{n-1}} |w\cdot u|^2d\widetilde{S}_{\mu,p}(K, u),
\end{equation}
for any  $w\in S^{n-1} $. Then, by \eqref{as},
	\begin{align}\label{pff1}
		h_{\Gamma_{\mu,p}K}(w) \nonumber 
		&=\left(\frac{1}{|\widetilde{V}_{\mu ,K}|}\int_{S^{n-1}}\left |w \cdot u\right |^{p} d \widetilde{S}_{\mu, p}(K,u)\right)^\frac{1}{p}\nonumber \\
		& \geq\left(\frac{1}{|\widetilde{V}_{\mu ,K}|}\int_{S^{n-1}}\left |w \cdot u\right |^{2} d \widetilde{S}_{\mu, p}(K,u)\right)^\frac{1}{p}\nonumber \\
		&= 1,
	\end{align}	
for each $w\in S^{n-1}$. Thus,  $\widetilde{E}_{\mu, p}K\subseteq \Gamma_{\mu,p}K$.

On the other hand, by \eqref{KK}, \eqref{Kp}, \eqref{eyh},  and \eqref{dcm1}, we have	
\begin{equation*}
		|\widetilde{V}_{\mu,K}|=\widetilde{V}_{\mu, -p}(K,K)=\widetilde{V}_{\mu, -p}(K,B^n)=\frac1n\int_{S^{n-1}} d \widetilde{S}_{\mu ,p}(K,u).
	\end{equation*}
Thus, by Jensen's inequality and \eqref{bn}, we further have 
	\begin{align}\label{pfff1}
		h_{\Gamma_{\mu,p}K}(w)  \nonumber 
		&= n^\frac{1}{p}\left(\frac{1}{n|\widetilde{V}_{\mu,K}|}\int_{S^{n-1}} |w\cdot u|^{p} d \widetilde{S}_{\mu ,p}(K,u)\right)^\frac{1}{p}\nonumber \\
			&\le n^\frac{1}{p}\left(\frac{1}{n|\widetilde{V}_{\mu,K}|}\int_{S^{n-1}} |w\cdot u|^{2} d \widetilde{S}_{\mu ,p}(K,u)\right)^\frac{1}{2}\nonumber \\
				&=n^{\frac{1}{p}-\frac{1}{2}},
	\end{align}	
for each $w\in S^{n-1}$. That is,  $\Gamma_{\mu,p}K	\subseteq n^{\frac{1}{p}-\frac{1}{2}}\widetilde{E}_{\mu, p}K$.
	
	When $2\leq p<\infty$, inequalities \eqref{pff1} and \eqref{pfff1} are reversed. In light of Problem  $\widetilde{M}_{\mu, \infty}$ and relation \eqref{e-10}, we find  that  $\widetilde E_{\mu, \infty}K$ is the L\"{o}wner ellipsoid of $\Gamma_{\mu,\infty}K=\mathrm{conv}\big\{\widetilde{K}\cup (-\widetilde{K})\big\}$. Consequently, the desired inclusion follows from the  L\"{o}wner inclusion \eqref{LK}.  
\end{proof}

The $L_p$ L\"{o}wner inclusion for the new dual $L_p$ John ellipsoid can be formulated as follows. In the case where $\mu$ is the Lebesgue measure on $\mathbb{R}^n$ and $K\in\mathcal {K}_{o}^n$, Theorem \ref{cc1-2} was established in \cite{YLW}.
\begin{thm}\label{cc1-2}
If  $1\leq p\leq \infty$ and $K\in \mathcal {S}^n_{o}$, then
	$$\widetilde{E}_{\mu, p}K\subseteq \Gamma_{\mu,p}K\subseteq n^{\frac{1}{p}-\frac{1}{2}}\widetilde{E}_{\mu, p}K, \quad 1\leq p\leq 2;$$
	$$
	n^{\frac{1}{p}-\frac{1}{2}}\widetilde{E}_{\mu, p}K \subseteq	\Gamma_{\mu,p}K \subseteq\widetilde{E}_{\mu, p}K ,\quad  2\leq p\leq \infty.
	$$
\end{thm}
\begin{proof}
	Without loss
	of generality, we  assume that  $\widetilde{E}_{\mu, p}K=\phi B^n$ for some
	$\phi\in \textrm{GL}(n)$. {
Since $(\mu^\phi)^{\phi^{-1}}=\mu$, it follows from \eqref{Ea}, applied to the pair $(\mu^\phi,K)$ and the map $\phi^{-1}$, that
}
\begin{equation*}
\widetilde{E}_{\mu^{\phi}, p}
\phi^{-1}K=\phi^{-1}\widetilde{E}_{\mu, p}
		K=\phi^{-1}\phi B^n=B^n.
	\end{equation*}
	{
Applying Lemma \ref{tl} to the pair $(\mu^\phi,\phi^{-1}K)$ yields
}
	$$\widetilde{E}_{\mu^{\phi}, p}\phi^{-1}K\subseteq  \Gamma_{\mu^{\phi}, p}\phi^{-1}K\subseteq
	n^{\frac{1}{p}-\frac{1}{2}}\widetilde{E}_{\mu^{\phi}, p}\phi^{-1}K,\quad
	1\leq p \leq 2;$$ $$n^{\frac{1}{p}-\frac{1}{2}}\widetilde{E}_{\mu^{\phi}, p}\phi^{-1}K\subseteq \Gamma_{\mu^{\phi}, p}\phi^{-1}K\subseteq \widetilde{E}_{\mu^{\phi}, p}\phi^{-1}K,\quad
	2 \leq p \leq \infty.$$
	Moreover, by \eqref{Ea} and
	\eqref{zxt}, we have
	\begin{equation*}
		\widetilde{E}_{\mu^{\phi}, p}\phi^{-1}K=\phi^{-1}\widetilde{E}_{\mu, p}
		K\quad \textrm{and}\quad\Gamma_{\mu^{\phi}, p}\phi^{-1}K =\phi^{-1}\Gamma_{\mu, p}
		K.
	\end{equation*}
	Consequently, we get
	$$\widetilde{E}_{\mu, p}K\subseteq \Gamma_{\mu,p}K\subseteq n^{\frac{1}{p}-\frac{1}{2}}\widetilde{E}_{\mu, p}K, \quad 1\leq p\leq 2;$$
	$$
	n^{\frac{1}{p}-\frac{1}{2}}\widetilde{E}_{\mu, p}K \subseteq	\Gamma_{\mu,p}K \subseteq\widetilde{E}_{\mu, p}K ,\quad  2\leq p\leq \infty.
	$$
\end{proof}


\begin{thebibliography}{16}

	\bibitem{AWHQ}
	 W. Ai and Q. Huang, \textit{The $L_p$ John ellipsoids for general measures}, Geom. Dedicata \textbf{217} (2023), Paper No.  17, 19 pp.
	
	\bibitem{Alonso-S-2}
D. Alonso-Guti\'{e}rrez and S. Brazitikos, \textit{Sections of
	convex bodies in John's and minimal surface area position}, Int. Math. Res. Not. (IMRN) \textbf{2023}  (2023), 243-297.	
	
	
	 
	\bibitem{Ball-91}
	K. Ball, \textit{Volume ratios and a reverse isoperimetric
		inequality}, J. London Math. Soc. \textbf{44} (1991), 351-359.

\bibitem{F-Barthe}
F. Barthe, \textit{On a reverse form of the Brascamp-Lieb inequality}, Invent. Math. \textbf{134} (1998), 335-361.

	\bibitem{Barthe}
	F. Barthe, \textit{An extremal property of the mean width of the
		simplex}, Math. Ann. \textbf{310} (1998), 685-693.
	


\bibitem{Barthe05}
F. Barthe, O. Guedon, S. Mendelson, and A. Naor,  \textit{A probabilistic approach to the geometry of the $l_p^n$
ball}, Ann. Probab. \textbf{33} (2005), 480-513. 
	
		
	
	\bibitem{BR}
	J. Bastero and M. Romance, \textit{Positions of convex bodies
		associated to extremal problems and isotropic measures}, Adv. Math.
	\textbf{184} (2004), 64-88.
	
	
	\bibitem{BLYZZ}
	K.J. B\"{o}r\"{o}czky, E. Lutwak, D. Yang,  G. Zhang, and Y. Zhao, \textit{The dual Minkowski problem for symmetric convex bodies}, Adv. Math. \textbf{356} (2019),
	106805.
	
	\bibitem{Firey}
	W.J. Firey, \textit{p-means of convex bodies}, Math. Scand.
	\textbf{10} (1962), 17-24.
	
	\bibitem{FLMZI}
	M. Fradelizi, D. Langharst, M. Madiman, and A. Zvavitch, \textit{Weighted Brunn-Minkowski theory
	I: on weighted surface area measures}, {
J. Math. Anal. Appl.
} \textbf{529} (2024), Paper No.  127519, 30 pp.

\bibitem{FLMZII}
 M. Fradelizi, D. Langharst, M. Madiman, and A. Zvavitch, \textit{Weighted Brunn-Minkowski theory
	{
II: Inequalities for Mixed Measures and Applications}, J. Anal. Math. (2026), doi:10.1007/s11854-026-0440-5; arXiv:2402.10314.
}
	
	
	
	
	\bibitem{Gardner}
	R.J. Gardner, Geometric Tomography, second edition, Cambridge
	University Press, New York, 2006.
	
	
	\bibitem{GWXY}
	R.J. Gardner, D. Hug, W. Weil, S. Xing, and D. Ye, \textit{General
		volumes in the Orlicz-Brunn-Minkowski theory and a related Minkowski
		problem I}, Calc. Var. Partial Differential Equations \textbf{58}
	(2019), Paper No. 12, 35 pp.
	
	
	\bibitem{GiaM}
	 A. Giannopoulos and V.D. Milman, \textit{Extremal problems and isotropic
		positions of convex bodies}, Israel J. Math. \textbf{117} (2000),
	29-60.
	
	
	\bibitem{Gia} 
	A. Giannopoulos and M. Papadimitrakis, \textit{Isotropic surface area measures}, Mathematika \textbf{46}
	(1999), 1-13.

\bibitem{GM} 	
M. Gromov and  V.D. Milman,  \textit{Generalization of the spherical isoperimetric inequality for uniformly
convex Banach Spaces}, Compos. Math. \textbf{62} (1987), 263-282. 	
	
	
	\bibitem{Gruber} 
	P.M. Gruber,  \textit{John and Loewner ellipsoids}, Discret. Comput. Geom. \textbf{46} (2011), 776-788.
	
	\bibitem{Henk} M. Henk, \textit{L\"{o}wner–John ellipsoids}, Doc. Math., Extra vol.: Optimization stories 2012, 95-106.
	
	\bibitem{Hos1} 
	J. Hosle, \textit{On the comparison of measures of convex bodies via
		projections and sections},  Int. Math. Res. Not. (IMRN) \textbf{2021} (2021), 
	13046-13074.
	
	
	
	
	\bibitem{HX}
	 J. Hu and G. Xiong, \textit{The logarithmic John ellipsoid},  Geom. Dedicata \textbf{197} (2018), 33-48.
	
	\bibitem{HXZ}
	 J. Hu, G. Xiong, and D. Zou, \textit{On mixed $L_p$ John ellipsoids}, Adv.
	Geom. \textbf{19} (2019),  297-312.
	
	\bibitem{HLYZ}  
	Y. Huang, E. Lutwak, D. Yang, and G. Zhang, \textit{
		Geometric measures in the dual Brunn-Minkowski theory and their
		associated Minkowski problems}, Acta Math.  \textbf{216} (2016),
	325-388.
	
	\bibitem{YXZ} 
	 Y. Huang, D. Xi, and Y. Zhao, \textit{The Minkowski problem in Gaussian
		probability space}, Adv. Math.  \textbf{385} (2021), Paper No.
	107769, 36 pp.
	
	\bibitem{J}
	F. John, \textit{Extremum problems with inequalities as subsidiary
		conditions}, in: Studies and Essays Presented to R. Courant on His
	60th Birthday, Interscience Publishers, Inc., New York, 1948, 187-204.

	\bibitem{KL} 
L. Kryvonos  and D. Langharst, \textit{Weighted Minkowski’s existence theorem and
projection bodies}, Trans. Amer.
Math. Soc. \textbf{376} (2023), 8447-8493.

	

	\bibitem{LRZ} 
	D. Langharst, M. Roysdon, and  A. Zvavitch,  \textit{General measure extensions of projection
		bodies}, Proc. London Math. Soc.  \textbf{125} (2022), {
1083-1129.
}
	
	\bibitem{Lew} 
	D.R. Lewis, \textit{Ellipsoids defined by Banach ideal norms}, Mathematika \textbf{26} (1979), 18-29.
	
	\bibitem{LHX}
	A.-J. Li, Q. Huang, and D. Xi, \textit{New sine ellipsoids and related volume inequalities}, Adv. Math. \textbf{353} (2019), 281-311.
	
	
	\bibitem{LWZ}
	X. Li, H. Wang, and J. Zhou, \textit{On $(p,q)$-John
		ellipsoids}(Chinese), Sci. Sin. Math. \textbf{50} (2020), 1-23.
	
	\bibitem{Lj} 
	J. Liu, \textit{The $L_p$-Gaussian Minkowski problem}. Calc. Var. Partial Differential Equations \textbf{61} (2022), Paper No. 28, 23 pp.
	
	
	\bibitem{Li} 
	G. Livshyts, \textit{An extension of Minkowski's theorem and its applications to questions about projections for measures}, Adv.
	Math. \textbf{356} (2019), Paper No. 106803, 40 pp.
	
	\bibitem{LM}
	 G. Livshyts, A. Marsiglietti, P. Nayar, and A. Zvavitch,
	\textit{On the Brunn-Minkowski inequality for general measures with
		applications to new isoperimetric-type inequalities}, Trans. Amer.
	Math. Soc. \textbf{369} (2017), 8725-8742.
	
	\bibitem{LuX}
	X. Lu and G. Xiong, 
	\textit{The $L_p$ John ellipsoids for negative indices}, Israel J. Math. \textbf{255} (2023), 155-176.
	
	
	\bibitem{L75}
	E. Lutwak, \textit{Dual mixed volumes}, Pac. J. Math. \textbf{58} (1975), 531-538. 
	
	\bibitem{L93}
	E. Lutwak, \textit{The Brunn-Minkowski-Firey theory I: Mixed volumes
		and the Minkowski Problem},  J. Differential Geom. \textbf{38}
	(1993), 131-150.
	
	\bibitem{L96}
	E. Lutwak, \textit{The Brunn-Minkowski-Firey theory II: Affine and Geominimal Surface Areas},  Adv. Math. \textbf{118}
	(1996), 244-294.
	
	
	\bibitem{LYZel}
	E. Lutwak, D. Yang, and G. Zhang, \textit{A new ellipsoid associated
		with convex bodies}, Duke Math. J. \textbf{104} (2000), 375-390.
	
	\bibitem{LYZCR}
	E. Lutwak, D. Yang, and G. Zhang, \textit{The Cramer-Rao inequality
		for star bodies}, Duke Math. J. \textbf{112} (2002), 59-81.
	
	
	\bibitem{LYZJ}
	E. Lutwak, D. Yang, and G. Zhang,\textit{ $L_p$ John ellipsoids},
	Proc. London Math. Soc. \textbf{90} (2005), 497-520.
	
	\bibitem{LYZ10}
	E. Lutwak, D. Yang, and G. Zhang, \textit{Orlicz centroid bodies}, J. Differential Geom.  \textbf{84} (2010), 365-387. 
	
	
	\bibitem{LYZ18}
	E. Lutwak, D. Yang, and G. Zhang,  \textit{$L_p$ dual curvature
		measures}, Adv. Math. \textbf{329} (2018), 85-132.

\bibitem{LZ97}
E. Lutwak and G. Zhang,  \textit{Blaschke-Santal\'{o} inequalities},
J. Differential Geom. \textbf{47} (1997), 1-16.

\bibitem{Lv}
   S. Lv,	\textit{A sharp dual $L_p$ John ellipsoid problem for $p \leq -n-1$}, Beitr. Algebra Geom. \textbf{60} (2019), 709-732.
	
	\bibitem{MWF}
	T. Ma, D. Wu, and Y. Feng, \textit{$(p, q)$-John Ellipsoids}, J.
	Geom. Anal. \textbf{31} (2021), 9597-9632.
	
	
	\bibitem{MP}
	V.D. Milman and A. Pajor, \textit{Isotropic position and inertia
		ellipsoids and zonoids of the unit ball of a normed n-dimensional
		space}, Geom. Aspects of Funct. Analysis (Lindenstrauss-Milman
	eds.), Lecture Notes in Math. \textbf{1376} (1989), 64-104.
	
	\bibitem{Mil} 
	E. Milman and L. Rotem, \textit{Complemented Brunn-Minkowski inequalities and isoperimetry for homogeneous and non-homogeneous measures}, Adv.
	Math. \textbf{262} (2014), 867-908.
	

	\bibitem{PettyS} 
	C.M. Petty, \textit{Surface area of a convex body under affine
		transformations}, Proc. Amer. Math. Soc. \textbf{12} (1961),
	824-828.
	
	\bibitem{RX} M. Roysdon and S. Xing, \textit{On $L_p$-Brunn-Minkowski type and
		$L_p$-isoperimetric type inequalities for measures}, Trans. Amer.
	Math. Soc. \textbf{374} (2021), 5003-5036.
	
	\bibitem{Schneider}
	R. Schneider, Convex bodies: the Brunn-Minkowski theory,
	Encyclopedia of Mathematics and its Applications, \textbf{151},
	Cambridge University Press, Cambridge, 2014.
	
	
	\bibitem{Wu}
	 D. Wu, \textit{A generalization of $L_p$-Brunn-Minkowski inequalities and
		$L_p$-Minkowski problems for measures}, Adv.  Appl. Math.
	\textbf{89} (2017), 156-183.
	
	\bibitem{Wu2}
	 D. Wu, \textit{Firey-Shephard problems for homogeneous measures}, J. Math. Anal. Appl. \textbf{458} (2018), 43-57.

	\bibitem{Wu3}
	D. Wu and  Z.-H. Bu, \textit{The measure-comparison problem for polar $(p,\mu)$-centroid bodies},  Adv.  Appl. Math.
	\textbf{137} (2022), Paper No.
	102332, 31 pp.
	
	\bibitem{YLW} 
	W. Yu, G. Leng, and D. Wu, \textit{Dual $L_p$ John
		ellipsoids}, Proc. Edinb. Math. Soc. \textbf{50} (2007), 737-753.
	
	
	\bibitem{Zhao}
	Y. Zhao, \textit{Existence of solutions to the even dual Minkowski problem}, J. Differential Geom. \textbf{110} (2018), 543-572.
	
	\bibitem{ZX} 
	D. Zou and G. Xiong, \textit{Orlicz-John ellipsoids}, Adv. Math. \textbf{265} (2014),
	132-168.
	
	\bibitem{ZX2}
	 D. Zou and G. Xiong, \textit{Orlicz-Legendre ellipsoids}, J. Geom. Anal.
	\textbf{26} (2016), 2474-2502.
	
	\bibitem{Zv}  A. Zvavitch,  \textit{The Busemann-Petty problem for arbitrary
		measures}, Math. Ann. \textbf{331} (2005), 867-887.
	
	
	
	
	
	
\end{thebibliography}
\end{document}